\documentclass[12pt]{amsart}
\usepackage{amsfonts,amssymb,amscd,amsmath,enumerate,verbatim}
\usepackage[latin1]{inputenc}
\usepackage{amscd}
\usepackage{tikz}
\usepackage{hyperref}
\hypersetup{colorlinks,linkcolor=blue ,citecolor=blue, urlcolor=blue}
\def\NZQ{\mathbb}               % the font for N,Z,Q,R,C

\def\ZZ{{\NZQ Z}}
\def\RR{{\NZQ R}}

\def\frk{\mathfrak}               % font for "Fraktur"

\def\Phi{{\frk N}}
\def\ab{{\bold a}}
\def\bb{{\bold b}}

\def\eb{{\bold e}}

\def\xb{{\bold x}}
\def\yb{{\bold y}}
\def\zb{{\bold z}}
\def\vb{{\bold v}}
\def\opn#1#2{\def#1{\operatorname{#2}}} % to make operators
\opn\chara{char} 
\opn\length{\ell} 
\opn\pd{pd} 
\opn\rk{rk}
\opn\projdim{proj\,dim} 
\opn\injdim{inj\,dim} 
\opn\rank{rank}
\opn\depth{depth} 
\opn\grade{grade} 
\opn\height{height}
\opn\embdim{emb\,dim} 
\opn\codim{codim}

\opn\Tr{Tr} 
\opn\bigrank{big\,rank}
\opn\superheight{superheight}
\opn\lcm{lcm}
\opn\trdeg{tr\,deg}%\emph{
\opn\reg{reg} 
\opn\lreg{lreg} 
\opn\ini{in} 
\opn\lpd{lpd}
\opn\size{size}
\opn\mult{mult}
\opn\dist{dist}
\opn\cone{cone}
\opn\lex{lex}
\opn\rev{rev}
\opn\div{div} \opn\Div{Div} \opn\cl{cl} \opn\Cl{Cl}
\opn\Spec{Spec} \opn\Supp{Supp} \opn\supp{supp} \opn\Sing{Sing}
\opn\Ass{Ass} \opn\Min{Min}
\opn\Ann{Ann} \opn\Rad{Rad} \opn\Soc{Soc}
\opn\Syz{Syz} \opn\Im{Im} \opn\Ker{Ker} \opn\Coker{Coker}
\opn\Am{Am} \opn\Hom{Hom} \opn\Tor{Tor} \opn\Ext{Ext}
\opn\End{End} \opn\Aut{Aut} \opn\id{id} \opn\ini{in}

\opn\nat{nat}
\opn\pff{pf}%   \pf exists already
\opn\Pf{Pf} \opn\GL{GL} \opn\SL{SL} \opn\mod{mod} \opn\ord{ord}
\opn\Gin{Gin}
\opn\Hilb{Hilb}\opn\adeg{adeg}\opn\std{std}\opn\ip{infpt}
\opn\Pol{Pol}
\opn\sat{sat}
\opn\Var{Var}
\opn\Gen{Gen}

\opn\aff{aff} \opn\con{conv} \opn\relint{relint} \opn\st{st}
\opn\lk{lk} \opn\cn{cn} \opn\core{core} \opn\vol{vol}
\opn\link{link} \opn\star{star}
\opn\gr{gr}

\def\Pc{{\mathcal P}}
\def\Qc{{\mathcal Q}}

\def\pot#1#2{#1[\kern-0.28ex[#2]\kern-0.28ex]}

\opn\dirlim{\underrightarrow{\lim}}
\opn\inivlim{\underleftarrow{\lim}}
\let\union=\cup

\def\Implies{\ifmmode\Longrightarrow \else
	\unskip${}\Longrightarrow{}$\ignorespaces\fi}
\def\implies{\ifmmode\Rightarrow \else
	\unskip${}\Rightarrow{}$\ignorespaces\fi}
\def\iff{\ifmmode\Longleftrightarrow \else
	\unskip${}\Longleftrightarrow{}$\ignorespaces\fi}

\let\:=\colon
\newtheorem{Theorem}{Theorem}[section]
\newtheorem{Lemma}[Theorem]{Lemma}
\newtheorem{Corollary}[Theorem]{Corollary}
\newtheorem{Proposition}[Theorem]{Proposition}

\newtheorem{Example}[Theorem]{Example}

\newtheorem{Definition}[Theorem]{Definition}

\newtheorem{Conjecture}[Theorem]{Conjecture}
\newtheorem{Question}[Theorem]{Question}

\opn\dis{dis}
\opn\height{height}
\opn\dist{dist}
\def\pnt{{\raise0.5mm\hbox{\large\bf.}}}

\opn\Lex{Lex}
\opn\conv{conv}
\opn\inter{int}

\begin{document}
%%%%%%%%%%%%%%%%%%%%%%%%%%%%%%%%%%%%%%%%
%% title
%%%%%%%%%%%%%%%%%%%%%%%%%%%%%%%%%%%%%%%%
\title{Cayley sums and Minkowski sums of  lattice polytopes}
\author[A. Tsuchiya]{Akiyoshi Tsuchiya}
\address{Akiyoshi Tsuchiya,
	Department of Information Science,
	Faculty of Science,
	Toho University,
	Miyama, Funabashi-shi, Chiba 274-8510, Japan} 
\email{akiyoshi@is.sci.toho-u.ac.jp}

\date{}
\thanks{The author was partially supported by KAKENHI 16J01549, 19K14505 and 22K13890.}

\begin{abstract}
	In this paper, we discuss the integer decomposition property for Cayley sums and Minkowski sums of lattice polytopes.
	In fact, we characterize when Cayley sums have the integer decomposition property in terms of Minkowski sums. 
	Moreover, by using this characterization, we consider when Cayley sums and Minkowski sums of $2$-convex-normal lattice polytopes have the integer decomposition property.
	 Finally, we also discuss the level property for Minkowski sums and Cayley sums.
\end{abstract} 

\maketitle
\section{Introduction}
A \textit{lattice polytope} is a convex polytope all of whose vertices have integer coordinates.
In the present paper, we discuss two algebraic properties of lattice polytopes, which are called the integer decomposition property (IDP) and the level property (see Section \ref{sec:pre} for the definitions).
IDP polytopes turn up in many fields of mathematics such as algebraic geometry, where they correspond to projectively normal embeddings of toric varieties, and commutative algebra, where they correspond to standard graded Cohen-Macaulay domains (see \cite{BG}).
Moreover, the integer decomposition property is particularly important in the theory and application of integer programming \cite[\S 22.10]{integer}. On the other hand, the level property is a generalization of the Gorenstein property, which gives important examples in combinatorial commutative algebra, mirror symmetry and tropical geometry (for details we refer to \cite{Batyrev,JK}), and it has only fairly recently been examined for certain classes of polytopes (e.g., \cite{Hig,HY,KO}).

Our interest is to determine when a lattice polytope is IDP or level.
For instance, every lattice polygon is IDP and level (cf. \cite{HH,HY}).
One of the most famous results on this problem is the following:
\begin{Theorem}[{\cite[Theorem 1.3.3]{BGT}}]
\label{orginal}
	Let $\Pc \subset \RR^N$ be a lattice polytope.
Then we obtain the following:
	\begin{enumerate}
	\item[{\rm (1)}] For any positive integer $n \geq \dim(\Pc)-1$, $n\Pc$ is IDP;
 \item[{\rm (2)}] For any positive integer $n \geq \dim(\Pc)+1$, $n\Pc$ is (IDP and) level of index $1$.
\end{enumerate}
\end{Theorem}

We recall two well-known constructions of lattice polytopes.
\begin{Definition}{\em
		Let $\Pc_1,\ldots,\Pc_m \subset \RR^N$ be lattice polytopes.
		The \textit{Minkowski sum} $\Pc_1+\cdots+\Pc_m$ of $\Pc_1$, $\ldots,\Pc_m$ is defined by
		\[
		\Pc_1+\cdots+\Pc_m=\{\ab_1+\cdots+\ab_m \in \RR^N \mid \ab_1 \in \Pc_1,\ldots, \ab_m \in \Pc_m \} \subset \RR^N.
		\]
		The \textit{Cayley sum} $\Pc_1*\cdots*\Pc_m$ of $\Pc_1$, $\ldots,\Pc_m$ is defined by
		\[\Pc_1*\cdots*\Pc_m=\textnormal{conv}(\{\eb_1\} \times \Pc_1 \union \cdots \union\{\eb_m\} \times \Pc_m ) \subset \RR^{m}\times \RR^{N} = \RR^{m+N},
		\]
		where $\eb_1,\ldots,\eb_m$ are the canonical unit coordinate vectors of $\RR^m$.
	}
\end{Definition}
%These two constructions often appear in several papers at the same time.

In Section \ref{sec:relation}, we discuss relationships between Minkowski sums and Cayley sums for the IDP and the level property.
Let $\Pc_1,\ldots,\Pc_m \subset \RR^N$ be lattice polytopes.
We say that the tuple $(\Pc_1,\ldots,\Pc_m)$ is \textit{IDP} if for any subset $\emptyset \neq I \subset [m]:=\{1,\ldots,m\}$,
the equation
%\begin{equation}
%\label{eq1}
\[
\left(\sum_{i \in I}\Pc_{i} \right) \cap \ZZ^N=\sum_{i \in I}(\Pc_i \cap \ZZ^N) % \tag{0.1}
%		\end{equation}
\]
 is satisfied.
When $m=2$, this notion is introduced and discussed in \cite{HH}.
In \cite{Oda}, Oda asked for which pair of lattice polytopes $\Pc, \Qc \subset  \RR^N$, $(\Pc,\Qc)$ is IDP.
The following theorem is the first main theorem of the present paper.
\begin{Theorem}
	\label{thm;relation}
	Let $\Pc_1,\ldots,\Pc_m \subset \RR^N$ be lattice polytopes.
	Then we obtain the following:
	\begin{enumerate}
		%\item 	If $C(\Pc_1,\ldots,\Pc_m)$ is normal, then 
		%for any $\emptyset \neq I \subset [m]:=\{1,\ldots,m\}$, $\sum_{i \in I}\Pc_i$ is normal  and we have $(\sum_{i \in I}\Pc_i) \cap \ZZ^N=\sum_{i \in I}(\Pc_i \cap \ZZ^N)$;
		\item [{\rm (1)}] The Cayley sum $\Pc_1*\cdots*\Pc_m$ is IDP if and only if each $\Pc_i$ is IDP and for any nonnegative integers $a_1,\ldots, a_m$, the tuple $(a_1 \Pc_1,\ldots, a_m \Pc_m)$ is IDP.
			In this case, the Minkowski sum $\sum_{i=1}^ma_i\Pc_i$ is IDP for any $a_1,\ldots,a_m$.
		\item[{\rm (2)}] If $\Pc_1*\cdots*\Pc_m$ is level of index $m$, then $\Pc_1+\cdots+\Pc_m$ is level of index $1$.
	\end{enumerate}
\end{Theorem}
%It is known {\cite[Theorem 2.6]{BN}} that $\Pc_1*\cdots*\Pc_m$ is Gorenstein of index $m$ if and only if $\Pc_1+\cdots+\Pc_m$ is Gorenstein of index $1$.
%However, the converse of Theorem \ref{thm;relation} (2) does not hold in general (Example \ref{counter}). 

%Next, we discuss when Minkowski sums and Cayley sums are IDP or level.
In \cite{Hig}, Higashitani considered when the Minkowski sums of dilated polytopes are IDP or level.
In fact, he showed the following:
\begin{Theorem}[{\cite[Theorem 2.1]{Hig}}]
	\label{Higa}
	Let $\Pc_1,\ldots,\Pc_m \subset \RR^N$ be lattice polytopes. For each $i$, let $n_i$ be a positive integer.
	Then we obtain the following:
	\begin{enumerate}
			\item[{\rm (1)}] If for each $i$, $n_i \geq \dim (\Pc_i)$, then $n_1\Pc_1+\cdots+n_m\Pc_m$ is IDP;
		\item[{\rm (2)}] If for each $i$, $n_i \geq \dim (\Pc_i)+1$, then $n_1\Pc_1+\cdots+n_m\Pc_m$ is (IDP and) level of index $1$.
	\end{enumerate}
\end{Theorem}

Therefore, Theorems \ref{thm;relation} and \ref{Higa} naturally lead us to consider the following question:
\begin{Question}
	\label{ques}
	Let $\Pc_1,\ldots,\Pc_m \subset \RR^N$ be lattice polytopes. For each $i$, let $n_i$ be a positive integer.
	\begin{enumerate}
		\item[{\rm (1)}] Is $n_1\Pc_1*\cdots*n_m\Pc_m$ IDP if $n_i \geq \dim (\Pc_i)$ for all $i=1,\ldots,m$ ?
		\item[{\rm (2)}] Is $n_1\Pc_1*\cdots*n_m\Pc_m$ level of index $m$ if  $n_i \geq \dim (\Pc_i)+1$ for all $i=1,\ldots,m$?
	\end{enumerate}
\end{Question}
In the present paper, we show that the answer to Question \ref{ques} (2) is positive (Corollary \ref{cor:leveldi}), but that to Question \ref{ques} (1) is negative in general (Corollary \ref{cor:IDPdi} and Example \ref{ex:IDPcounter}).  
To address these questions, we focus on $2$-convex-normal lattice polytopes and $2$-convex-level lattice polytopes.
In particular, we determine when Minkowski sums and Cayley sums of $2$-convex-normal (resp. $2$-convex-level) lattice polytopes are IDP (resp. level) (Theorems \ref{thm:2cn} and \ref{thn:2cl}).
%A lattice polytope $\Pc \subset \RR^N$ is called \textit{2-convex-level} if $\inter(2\Pc)=\inter(\Pc)+(\Pc \cap \ZZ^N)$.
%Note that 2-convex-level lattice polytopes are always level (Proposition \ref{prop:2cl}).
%By discussing when a dilated polytope is 2-convex-normal (resp. 2-convex-level) (Lemmas \ref{lem:2cn} and \ref{lem:2cl}) and when the Minkowski sum and the Cayley sum of 2-convex-normal (resp. 2-convex-level) lattice polytopes are IDP (resp. level) (Theorems \ref{thm:2cn} and \ref{thn:2cl}),
%we give an answer to Question \ref{ques}.
%
%The present paper is organized as follows:
%In Section 1, we will discuss relations between Minkowski sums and Cayley sums.
%In particular, we will prove Theorem \ref{thm;relation}.
%In Section 2, we will give an answer to Question \ref{ques} (1).
%In particular, we will prove Theorem \ref{thm:2cn}.
%Finally, in Section 3, we will give an answer to Question \ref{ques} (2).
%In particular, we will prove Theorem \ref{thn:2cl}.
%Higashitani proved the following:
%\begin{Theorem}[{\cite[Theorem 2.1]{Hig}}]
%	\label{Higashitani}
%	Let $\Pc_1,\ldots,\Pc_m \subset \RR^N$ be lattice polytopes.
%	Then we obtain the followings:
%	\begin{enumerate}
%		\item For any positive integers $n_1,\ldots,n_m$ with $n_i \geq \dim \Pc_i$ for each $i$, $n_1\Pc_1+\cdots + n_m \Pc_m$ is IDP;
%		\item For any positive integers $n_1,\ldots,n_m$ with $n_i \geq \dim \Pc_i+1$ for each $i$, $n_1\Pc_1+\cdots + n_m \Pc_m$ is level of index  $1$.
%	\end{enumerate}
%\end{Theorem}

\section{Preliminaries}
\label{sec:pre}
In this section, we give the definitions of IDP polytopes and level polytopes, and recall their properties.
Let $\Pc \subset \RR^N$ be a lattice polytope and let $\dim(\Pc)$ denote the dimension of $\Pc$. 
We say that a lattice polytope $\Pc$ has \textit{the integer decomposition property} if for each integer $n \geq 1$, 
\[
n\Pc \cap \ZZ^N=((n-1)\Pc \cap \ZZ^N)+(\Pc \cap \ZZ^N), 
\]
where $n\Pc$ is the $n$th dilated polytope of $\Pc$, i.e., $n\Pc=\{n\xb : \xb \in \Pc \}$.
A lattice polytope which has the integer decomposition property is called \textit{IDP}.

Let us explain a connection between IDP polytopes and commutative algebras.
Let $K$ be a field.
Given a graded noetherian commutative ring $A=\oplus_{i =0}^{\infty}A_i$ with $A_0=K$, we say that $A$ is {\it standard graded} if $A=K[A_1]$, i.e., $A$ is generated by $A_1$ as a $K$-algebra and {\it semi-standard graded} if $A$ is finitely generated as a $K[A_1]$-module.
We associate to a lattice polytope $\Pc$ a semi-standard graded $K$-algebra.
Let $K[{\bf X}^{\pm 1}, T]=K[X_1^{\pm1}, \ldots, X_N^{\pm 1}, T]$ be a subring of the the Laurent polynomial ring in $N+1$ variables over $K$.
We define the $K$-algebra $K[\Pc]$ as follows:
\[
K[\Pc]=K[{\bf X}^{\ab}T^n : \ab \in n\Pc \cap \ZZ^N, n \in \ZZ_{\geq 0}] \subset K[{\bf X}^{\pm 1}, T],
\]
where for a lattice point $\ab=(a_1,\ldots,a_N) \in \ZZ^N$, ${\bf X}^{\ab}T^n=X_1^{a_1}\cdots X_N^{a_N}T^n$ denotes a Laurent monomial in $K[{\bf X}^{\pm 1}, T]$.
It is known that $K[\Pc]$ is a semi-standard graded normal Cohen-Macaulay domain of dimension $\dim(\Pc)+1$ by setting $\deg({\bf X}^{\ab}T^n)=n$.
Moreover, $K[\Pc]$ is standard graded if and only if $\Pc$ is IDP.
We call this graded $K$-algebra $K[\Pc]$ the {\it Ehrhart ring} of $\Pc$.
We refer the reader to \cite{BG} for the detailed information about Ehrhart rings.

Theorem \ref{orginal} implies
that ``large polytopes" are IDP.
The following result also says the same thing.
\begin{Theorem}[{\cite[Corollary 2.7]{HH}}]
	\label{thm:hh}
Let $\Pc \subset \RR^N$ be a lattice polytope of dimension $d$.
If every edge of $\Pc$ has lattice length $\geq 2d(d+1)$, then $P$ is IDP.
\end{Theorem}
This theorem follows from the fact that  $2$-convex-normal lattice polytopes are IDP ({\cite[Corollary 2.6]{HH}}).
A lattice polytope $\Pc \subset  \RR^N$  is called \textit{$2$-convex-normal} if 
\[2\Pc=(\Pc \cap \ZZ^N)+\Pc.\]
It is known that lattice polytopes of dimension $d$ each of whose edges has lattice length $\geq 2d(d+1)$ are $2$-convex-normal (\cite[Theorem 1.2]{Gub}).
Hence Theorem \ref{thm:hh} follows from this fact.

Next, we recall what level polytopes are.
For a subset  $A \subset \RR^N$, let $\inter(A)$ denote the relative interior of $A$ with respect to the affine subspace of $\RR^N$ spanned by $A$.
We say that a lattice polytope $\Pc$ is \textit{level} of index $r$,
if $r = \min \{t \in \ZZ_{>0} : \inter(t\Pc) \cap \ZZ^N \neq \emptyset  \}$ and for each integer $n \geq r$, 
\[
\inter(n\Pc) \cap \ZZ^N=(\inter(r\Pc) \cap \ZZ^N)+((n-r)\Pc \cap \ZZ^N).
\]
In particular, if $\Pc$ is level of index $r$ and $|\inter(r\Pc) \cap \ZZ^N|=1$,
then $\Pc$ is called \textit{Gorenstein} of index $r$.

Now, we review a connection between level polytopes and commutative algebras.
Let $R$ be a Cohen-Macaulay graded ring with canonical module $\omega_R$.
Then the number 
\[
a(R)=-\min\{
i : (\omega_{R})_i \neq 0
\}
\]
is called the {\it $a$-invariant} of $R$ (see \cite{BH} for the definition).
It then follows that $a(K[\Pc])=-\min \{t \in \ZZ_{>0} : \inter(t\Pc) \cap \ZZ^N \neq \emptyset  \}$.
We say that $R$ is level if the canonical module $\omega_R$ of $R$ is generated by elements of the same degree.
The notion of level rings was introduced by Stanley \cite{Sta77}. By virtue of Danilov \cite{Dan} and Stanley \cite{Sta},
we know that the Ehrhart ring $K[\Pc]$ of $\Pc$ is level (of $a$-invariant $-r$) if and only if $\Pc$ is level (of index $r$).

As an analogy of $2$-convex-normal polytopes, we define $2$-convex-level polytopes.
A lattice polytope $\Pc \subset \RR^N$ is called \textit{2-convex-level} if \[\inter(2\Pc)=\inter(\Pc)+(\Pc \cap \ZZ^N).\]
Like as $2$-convex-normal lattice polytopes, $2$-convex-level lattice polytopes are always level.
%Note that 2-convex-level lattice polytopes are always level (Proposition \ref{prop:2cl}).
\begin{Proposition}
	\label{prop:2cl}
	Let $\Pc \subset \RR^N$ be a 2-convex-level lattice polytope.
	Then $\Pc$ is level.
\end{Proposition}

\begin{proof}
Let $r = \min \{t \in \ZZ_{>0} : \inter(t\Pc) \cap \ZZ^N \neq \emptyset  \}$.
	Fix a positive integer $n \geq r+1$ and take an interior lattice point $\xb$ in $n \Pc$.
	Since for lattice polytopes $\Pc_1,\ldots,\Pc_m \subset \RR^N$, ${\rm int}(\Pc_1+\cdots+\Pc_m)={\rm int}(\Pc_1)+\cdots+{\rm int}(\Pc_m)$ (cf. \cite[Section 1]{Convex}) and since $\Pc$ is 2-convex-level, we obtain
	\[
	{\rm int}(n\Pc)={\rm int}(r\Pc)+\underbrace{(\Pc \cap \ZZ^N)+\cdots+(\Pc \cap \ZZ^N)}_{n-r}.
	\]
	Hence there exist $\xb_1 \in {\rm int}(r\Pc)$ and $\xb_2,\ldots,\xb_{n-r+1} \in \Pc \cap \ZZ^N$ such that $\xb=\xb_1+\xb_2+\cdots+\xb_{n-r+1}$.
	In particular, from $\xb \in \ZZ^N$, one has $\xb_1 \in {\rm int}(r\Pc) \cap \ZZ^N$.
	Therefore, since $\xb_2+\cdots+\xb_{n-r+1} \in (n-r)\Pc \cap \ZZ^N$, $\Pc$ is level of index $r$.
\end{proof}

\section{Relationships between Minkowski sums and Cayley sums}
\label{sec:relation}
In this section, we discuss relationships between Minkowski sums and Cayley sums.
In particular, we prove Theorem \ref{thm;relation}.
%First, we consider a relation between  points in Minkowski sums and that in Cayley sums.
The following lemma is known as the Cayley trick (cf. \cite[Section 9]{triangulation}).
\begin{Lemma}
	\label{decomp1}
	Let $\Pc_1,\ldots,\Pc_m \subset \RR^N$ be convex polytopes.
	Then for any nonnegative real numbers $a_1,\ldots,a_m$, we have
	$$(a_1+\cdots+a_m)(\Pc_1*\cdots*\Pc_m) \cap 	
	\left( \left\{\sum_{i=1}^{m}a_i \eb_i\right\} \times \RR^{N} \right)
	=\left\{\sum_{i=1}^{m}a_i \eb_i\right\} \times \sum_{i=1}^{m} a_i \Pc_i$$
	and
		$${\rm int}((a_1+\cdots+a_m)(\Pc_1*\cdots*\Pc_m)) \cap 	
	\left( \left\{\sum_{i=1}^{m}a_i \eb_i\right\} \times \RR^{N} \right)
	=\left\{\sum_{i=1}^{m}a_i \eb_i\right\} \times {\rm int} \left(\sum_{i=1}^{m} a_i \Pc_i \right).$$

\end{Lemma}

Now, we prove Theorem \ref{thm;relation}.
\begin{proof}[Proof of Theorem \ref{thm;relation}]
	(1) Suppose that $\Pc_1*\cdots*\Pc_m$ is IDP. Since $\Pc_1,\ldots,\Pc_m$ are faces of $\Pc_1*\cdots*\Pc_m$, each $\Pc_i$ is IDP.
	Fix nonnegative integers $a_1,\ldots,a_m$ and take $\xb \in (n\sum_{i=1}^m a_i\Pc_i) \cap \ZZ^N$ with a positive integer $n$.
	Then by Lemma \ref{decomp1}, we have 
	$$\left(n\sum_{i=1}^m a_i \eb_i,\xb\right) \in n(a_1+\cdots+a_m)(\Pc_1*\cdots*\Pc_m) \cap  \left(\left\{n\sum_{i=1}^m a_i \eb_i \right\} \times \RR^N \right).$$
	Since $\Pc_1*\cdots*\Pc_m$ is IDP, we can write\[
	\left(n\sum_{i=1}^m a_i \eb_i,\xb \right) =\sum_{k=1}^{n}\sum_{j=1}^{a_1+\cdots+a_m}\xb_{k,j},\] where each $\xb_{k,j} \in (\Pc_1*\cdots*\Pc_m) \cap \ZZ^{m+N}.$
	In particular, from 
	\[
	(\Pc_1*\cdots*\Pc_m) \cap \ZZ^{m+N}= \bigcup_{i=1}^m (\{\eb_i\} \times (\Pc_i \cap \ZZ^N)),
	\]
	 we may assume that for any $1 \leq k \leq n$, $1 \leq i \leq m$ and $1 \leq j \leq a_i$, $\xb_{k,a_1+\cdots+a_{i-1}+j} \in \{\eb_i\} \times (\Pc_i \cap \ZZ^N)$.
	Hence it follows from \[\yb_{k,i}=\sum_{j=1}^{a_i} \xb_{k,a_1+\cdots+a_{i-1}+j} \in \{a_i \eb_i\} \times (a_i \Pc_i \cap \ZZ^N)
	\]
	for each $i$
	that the tuple $(a_1\Pc_1,\ldots,a_m \Pc_m)$ is IDP.
	Moreover, we can write 
	\[\left(n\sum_{i=1}^m a_i \eb_i,\xb \right) = \yb_1+\cdots+\yb_{n},\] where each $\yb_k=\yb_{k,1}+\cdots+\yb_{k,m} \in  \{\sum_{i=1}^m a_i \eb_i\} \times (\sum_{i=1}^{m} a_i \Pc_i \cap \ZZ^{N})$.
	This implies that the Minkowski sum $\sum_{i=1}^{m} a_i \Pc_i$ is IDP.
	
	Conversely, suppose that each $\Pc_i$ is IDP and for any nonnegative integers $a_1,\ldots, a_m$, the tuple $(a_1 \Pc_1,\ldots, a_m \Pc_m)$ is IDP.
	Take $\xb \in n(\Pc_1*\cdots*\Pc_m) \cap \ZZ^{m+N}$ with some positive integer $n$.
	It then follows from Lemma \ref{decomp1} that there are nonnegative integers $a_1,\ldots,a_m$ with $a_1+\cdots+a_m=n$ and $\yb \in (\sum_{i=1}^{m} a_i \Pc_i)\cap \ZZ^N$ such that $(\sum_{i=1}^m a_i \eb_i, \yb)=\xb$. 
	Since the tuple $(a_1 \Pc_1,\ldots, a_m \Pc_m)$ is IDP, there are $\yb_1,\ldots,\yb_m \in \ZZ^N$ such that each $\yb_i \in a_i\Pc_i \cap \ZZ^N$ and $\yb=\yb_1+\cdots+\yb_m$. Moreover since each $\Pc_i$ is IDP, for any $1 \leq i \leq m$, there are $\yb^{(i)}_1,\ldots,\yb^{(i)}_{a_i} \in \Pc_i \cap \ZZ^N$ such that $\yb_i=\yb^{(i)}_1+\cdots+\yb^{(i)}_{a_i}$.
	Therefore, $\Pc_1*\cdots*\Pc_m$ is IDP since each $(\eb_i, \yb^{(i)}_j) \in \Pc_1*\cdots*\Pc_m \cap \ZZ^{m+N}$ and $\xb=\sum_{i=1}^{m}\sum_{j=1}^{a_i}(\eb_i, \yb^{(i)}_j)$.

	(2)
	Every interior lattice point of $m(\Pc_1*\cdots*\Pc_m)$ belongs to $\{\eb_1+\cdots+\eb_m\} \times \ZZ^{N}$.
	Hence by Lemma  \ref{decomp1}, we know that $\Pc_1+\cdots+\Pc_m$ has interior lattice points.
	Let $\ab \in \text{int}(n(P_1+\cdots+P_m)) \cap \ZZ^N$ with a positive integer $n$.
	Then by Lemma \ref{decomp1}, one has $$(n(\eb_{1}+\cdots+\eb_m),\ab) \in \text{int}(nm(\Pc_1*\cdots*\Pc_m)) \cap  (\{n(\eb_{1}+\cdots+\eb_m)\} \times \RR^{N}).$$
	Since $\Pc_1*\ldots*\Pc_m$ is level of index $m$,
	we can write $(n(\eb_{1}+\cdots+\eb_m),\ab)=\bb_1+\bb_2$, where $\bb_1 \in \text{int}(m(\Pc_1*\cdots*\Pc_m)) \cap \ZZ^{m+N}$ and $\bb_2 \in  (n-1)m(\Pc_1*\cdots*\Pc_m)\cap \ZZ^{m+N}$.
	Then  by Lemma \ref{decomp1}, $\bb_1 \in \{\eb_{1}+\cdots+\eb_m\} \times \textnormal{int}(\Pc_1+\cdots+\Pc_m)$
	and $\bb_2 \in \{(n-1)(\eb_{1}+\cdots+\eb_m)\} \times (n-1)(\Pc_1+\cdots+\Pc_m)$.
	Hence $\Pc_1+\cdots+\Pc_m$ is level of index $1$, as desired.
\end{proof}

From the following example we know that 
$(a_1 \Pc_1,\ldots,a_m \Pc_m)$ may not be IDP for some $a_1,\ldots,a_m$ even if each of $\Pc_1,\ldots,\Pc_m$ is IDP and the tuple $(\Pc_1,\ldots,\Pc_m)$ is IDP.

\begin{Example}\label{ex:cn-counter}
	{\rm 
	Let $\Pc, \Qc \subset \RR^6$ be the lattice polytopes with
	\[
	\Pc = {\rm conv} \{ \eb_1+\eb_2,\eb_3+\eb_4,\eb_5+\eb_6\} \mbox{ and } \Qc={\rm conv} \{ \eb_1+\eb_6,\eb_2+\eb_3,\eb_4+\eb_5\}.
	\]
	Then both $\Pc$ and $\Qc$ are IDP, and the pair $(\Pc,\Qc)$ is IDP. Since the lattice point $(1,1,1,1,2,2) \in (2\Pc+2\Qc) \cap \ZZ^6$ does not belong to $(2\Pc \cap \ZZ^5)+(2\Qc \cap \ZZ^5)$, $(2\Pc, 2\Qc)$ is not IDP. 
	Moreover, $\Pc + \Qc$ and $\Pc * \Qc$ are not IDP.
	}
\end{Example}

Now, we discuss applications of Theorem \ref{thm;relation} (1) to Oda's question.
%focus on the integer decomposition property of a pair of polytopes.
A lattice polytope $\Pc$ is called \textit{smooth} if the edge directions at every vertex of $\Pc$ form a lattice basis.
The following conjecture by Oda is well-known (\cite{Oda}).

\begin{Conjecture}[Oda Conjecture]
	\label{conj:oda}
	Let $\Pc, \Qc \subset \RR^N$ be lattice polytopes. If $\Pc$ is smooth and the normal fan of $\Qc$ coarsens that of $\Pc$, then $(\Pc, \Qc)$ is IDP.
\end{Conjecture}
In particular, Oda conjectured that every smooth polytope is IDP.
Let $\Pc \subset \RR^N$ be a smooth polytope and $\Qc \subset \RR^N$ an IDP polytope whose normal fan coarsens the normal fan of $\Pc$.
Then for any positive integers $a,b$, $a \Pc$ is smooth and the normal fan of $\Pc$ (resp. $\Qc$) coincide with that of $a\Pc$ (resp. $b\Qc$).
Hence if Conjecture \ref{conj:oda} holds, then $\Pc$ is IDP and for any nonnegative integers $a,b$, $(a \Pc, b \Qc)$ is IDP.
Therefore, from Theorem \ref{thm;relation} (1) Conjecture \ref{conj:oda} implies the following conjecture.
\begin{Conjecture}
		Let $\Pc, \Qc \subset \RR^N$ be lattice polytopes. If $\Pc$ is smooth, $\Qc$ is IDP, and the normal fan of $\Qc$ coarsens that of $\Pc$, then $\Pc*\Qc$ is IDP.
\end{Conjecture}

Conjecture \ref{conj:oda} holds when $N=2$. Moreover, in this case, we do not need the smoothness assumption on $\Pc$.
In fact,
\begin{Theorem}[{\cite[Theorem 1.1]{HNP}}]
	Let $\Pc,\Qc \subset \RR^2$ be lattice polygons such that the normal fan of $\Qc$ coarsens that of $\Pc$.
	Then $(\Pc,\Qc)$ is IDP.
\end{Theorem}

Hence we obtain the following.
\begin{Corollary}
	Let $\Pc,\Qc \subset \RR^2$ be lattice polygons such that the normal fan of $\Qc$ coarsens that of $\Pc$.
	Then $\Pc*\Qc$ is IDP.
\end{Corollary}
Recently, this result has been generalized by Codenotti and Santos \cite{CS}.

Next, we discuss the converse of Theorem \ref{thm;relation} (2).
Recall the following result on the Gorenstein property of Minkowski sums and Cayley sums.
\begin{Theorem}[{\cite[Theorem 2.6]{BN}}]
	Let $\Pc_1,\ldots,\Pc_m \subset \RR^N$ be lattice polytopes with $\dim(\Pc_1+\cdots+\Pc_m)=N$.
	Then $\Pc_1*\cdots*\Pc_m$ is Gorenstein of index $m$ if and only if  $\Pc_1+\cdots+\Pc_m$ is Gorenstein of index $1$.
\end{Theorem}
Let us note that Gorenstein polytopes are level.
However, the converse of Theorem \ref{thm;relation} (2) does not hold in general.
\begin{Example}
	\label{counter}
	{\em
		Let $\Pc_1$ be the line segment from $(1,0)$ to $(0,1)$ and 
		$\Pc_2$ the line segment from $(1,1)$ to $(-h, -nh)$ with positive integers $h$ and $n$.
		Then since the dimension of $\Pc_1+\Pc_2$ is $2$, $\Pc_1+\Pc_2$ is level of index $1$.
		However, from \cite[Theorem 4.5]{HY}, we know that $\Pc_1*\Pc_2$ is not level.
		Hence the converse of Theorem \ref{thm;relation} (2) does not hold in general.
	}
\end{Example}

\section{Minkowski sums and Cayley sums of $2$-convex-normal lattice polytopes}
In this section, we will give an answer to Question \ref{ques} (1).
First, we discuss the integer decomposition property for Minkowski sums and Cayley sums of $2$-convex-normal lattice polytopes.
\begin{Theorem}
	\label{thm:2cn}
	Let $\Pc_1,\ldots,\Pc_m \subset \RR^{N}$ be $2$-convex-normal lattice polytopes. Then we obtain the following:
	\begin{enumerate}
		\item[{\rm (1)}] The Minkowski sum $\Pc_1+\cdots+\Pc_m$ is IDP;
		\item[{\rm (2)}] The Cayley sum $\Pc_1*\cdots*\Pc_m$ is IDP if (and only if) the tuple $(\Pc_1,\ldots,\Pc_m)$ is IDP.
	\end{enumerate}
\end{Theorem}
\begin{proof}
	(1) Given a lattice point  $\xb \in n(\Pc_1+\cdots+\Pc_m)$ with a positive integer $n$,
	there exist $m$ points $\xb_1,\ldots,\xb_m$ such that for each $i$, $\xb_i \in n\Pc_i$ and $\xb=\xb_1+\cdots+\xb_m$.
	Moreover, since each $\Pc_i$ is $2$-convex-normal, for each $i$, we can write
	$\xb_i=\yb_{i1}+\cdots+\yb_{in}$,
	where $\yb_{i1} \in \Pc_i$ and for any $2 \leq j \leq n$, $\yb_{ij} \in \Pc_i \cap \ZZ^N$.
	For $1 \leq j \leq n$, set $\zb_j=\yb_{1j}+\cdots+\yb_{mj}$.
	Then $\zb_2,\ldots,\zb_n$ are lattice points in $\Pc_1+\cdots+\Pc_m$.
	Moreover since $\xb$ is a lattice point,  $\zb_1$ must be a lattice point in $\Pc_1+\cdots+\Pc_m$.
	Hence $\Pc_1+\cdots+\Pc_m$ is IDP.
	
	(2) Every $2$-convex normal lattice polytope is IDP. Hence thanks to Theorem \ref{thm;relation} (1), it suffices to show that if $(\Pc_1,\ldots,\Pc_m)$ is IDP,
	then for any nonnegative integers $a_1,\ldots,a_m$, $(a_1\Pc_1,\ldots,a_m\Pc_m)$ is IDP.
	%In order to prove this, we use the induction on $m$.
	%When $m=1$, this is clear.
	%Assume that $m \geq 2$.
	Let $\xb \in (a_1\Pc_1+\cdots+a_m\Pc_m) \cap \ZZ^{N}$.
	Then we can write 
	$$\xb=\xb_1+\cdots+\xb_m,$$
	where $\xb_i \in a_i \Pc_i$.
	Since each $\Pc_i$ is $2$-convex-normal, $\xb_i$ can be written as
	$$\xb_i=\yb^{(i)}+\yb^{(i)}_1+\cdots+\yb^{(i)}_{a_i-1},$$
	where $\yb^{(i)} \in \Pc_i$ and $\yb^{(i)}_1,\ldots,\yb^{(i)}_{a_i-1} \in \Pc_i \cap \ZZ^N$. 
	Then $\yb^{(1)}+\cdots+\yb^{(m)}$ is a lattice point since both $\xb$ and $\sum_{i=1}^m\sum_{j=1}^{a_i-1} \yb_{j}^{(i)}$ are lattice points.
	In particular, $\yb^{(1)}+\cdots+\yb^{(m)} \in (\Pc_1+\cdots+\Pc_m) \cap \ZZ^N$.
	Since $(\Pc_1,\ldots,\Pc_m)$ is IDP, without loss of generality, 
	we may assume that for each $i$$, \yb^{(i)} \in \Pc_i \cap \ZZ^N$.
	Hence it follows that $\xb_i \in a_i \Pc_i \cap \ZZ^N$. Therefore, $(a_1 \Pc_1, \ldots, a_m \Pc_m)$ is IDP. 
	\end{proof}

Next, we see when a dilated polytope is $2$-convex-normal.
\begin{Lemma}
	\label{lem:2cn}
	Let $\Pc \subset \RR^N$ be a lattice polytope.
	Then for any integer $n \geq \dim(\Pc)$, 
	$n \Pc$ is $2$-convex-normal.
\end{Lemma}

\begin{proof}
	Set $d=\dim(\Pc)$.
	When $d=0$, this is clear.
	We assume that $d \geq 1$.
	Let $\ab \in 2n\Pc$.
	By Carath\'{e}odory's Theorem (cf. {\cite[Corollary 7.1j]{integer}}), there exist $d+1$ affinely independent vertices $\vb_0,\ldots,\vb_d$ of $\Pc$ such that $\ab=\sum_{i=0}^{d}\lambda_i\vb_i$, where $\lambda_i \geq 0$ and $\sum_{i=0}^{d}\lambda_i=2n$.
	%Since $2n \geq d+1$, there exists an index $j$ such that $\lambda_j \geq 1$.
	%Hence $\ab=(\sum_{i=0}^{d}\lambda_i\vb_i-\vb_j)+\vb_j$ and $(\sum_{i=0}^{d}\lambda_i\vb_i-\vb_j) \in (2n-1)\Pc$.
	Since $\sum_{i=0}^d (\lambda_i- \lfloor \lambda_i \rfloor) < d+1 \leq n+1$, we obtain $\sum_{i=0}^d \lfloor \lambda_i \rfloor > n-1$. Hence $\sum_{i=0}^d\lfloor \lambda_i \rfloor \geq n$.
	 Therefore, there exist nonnegative integers $\mu_0,\ldots,\mu_d$ such that $\lambda_i-\mu_i \geq 0$ for any $i$ and $\sum_{i=0}^d \mu_i = n$. 
	 Then $\ab$ can be written as
	\[
	\ab=\sum_{i=0}^d(\lambda_i-\mu_i)\vb_i+\sum_{i=0}^{d}\mu_i\vb_{i}.
	\]
	Since $\sum_{i=0}^d(\lambda_i-\mu_i)=n$ and $\lambda_i - \mu_i \geq 0$ for each $i$,  one has $\sum_{i=0}^d(\lambda_i-\mu_i)\vb_i \in n\Pc$. Moreover, since  $\sum_{i=0}^d \mu_i = n$ and each $\mu_i$ is a nonnegative integer, we obtain $\sum_{i=0}^{d}\mu_i\vb_{i} \in n\Pc \cap \ZZ^N$.
	Hence $\ab \in n\Pc +(n\Pc \cap \ZZ^N)$. 
	Therefore, $n\Pc$ is $2$-convex-normal.
\end{proof}
The bound $n \geq \dim (\Pc)$ in Lemma \ref{lem:2cn} is optimal.

\begin{Example}
	{\rm 
	Let $\Pc \subset \RR^2$ be the lattice triangle with vertices $(0,0),(1,0),(0,1)$. Then $\Pc$ is not $2$-convex-normal. Indeed, a point $\xb=(2/3,2/3) \in 2 \Pc$ cannot be written as $\xb=\ab+\bb$ with $\ab \in \Pc \cap \ZZ^2$ and $\bb \in \Pc$.
	Since $\dim \Pc -1 =1$, $(\dim \Pc -1 )\Pc$ is not $2$-convex-normal. 
	}
\end{Example}

From Theorem \ref{thm:2cn} and Lemma \ref{lem:2cn} we obtain the following corollary:
\begin{Corollary}
	\label{cor:IDPdi}
	Let $\Pc_1,\ldots,\Pc_m \subset \RR^N$ be lattice polytopes. For each $i$, let $n_i$ be a positive integer with $n_i \geq \dim (\Pc_i)$.
	Then we obtain the following:
	\begin{enumerate}
		\item[{\rm (1)}] {\rm (\cite{Hig})} $n_1\Pc_1+\cdots+n_m\Pc_m$ is IDP;
		\item[{\rm (2)}] 
		$n_1\Pc_1*\cdots*n_m\Pc_m$ is IDP if (and only if) $(n_1\Pc_1,\ldots,n_m\Pc_m)$ is IDP. 
	\end{enumerate}
\end{Corollary}

Since each of the polytopes $\Pc$ and $\Qc$ in Example \ref{ex:cn-counter} is of dimension $2$, the bound $n_i \geq \dim (\Pc_i)$ in Corollary \ref{cor:IDPdi} is optimal.

Now, we see that the answer to Question \ref{ques} (1) is negative in general.
\begin{Example}
	\label{ex:IDPcounter}
	{\em 
	Let $\Pc_1 \subset \RR^2$ be the line segment from $(0,0)$ to $(1,2)$ and $\Pc_2 \subset \RR^2$ the line segment from $(0,0)$ to $(1,0)$.
	Then for any positive integers $n_1,n_2$, $(n_1\Pc_1,n_2\Pc_2)$ is not IDP.
	Indeed, one has $(1,1) \in (n_1\Pc_1+n_2\Pc_2) \cap \ZZ^2$.
	On the other hand, since 
	\[
	(n_1 \Pc_1 \cap \ZZ^2)+(n_2 \Pc_2 \cap \ZZ^2)=
	\{
	(a+b,2a) : a,b \in \ZZ, 0 \leq a \leq n_1, 0 \leq b \leq n_2
	\},
	\]
	we obtain $(1,1) \notin (n_1 \Pc_1 \cap \ZZ^2)+(n_2 \Pc_2 \cap \ZZ^2)$.
	Hence $(n_1\Pc_1,n_2\Pc_2)$ is not IDP.
	Thus, it follows from Corollary  \ref{cor:IDPdi}  that $n_1\Pc_1*n_2\Pc_2$ is not IDP.
	Therefore, the answer to Question \ref{ques} (1) is negative in general.
}
\end{Example}

Finally, we give another corollary of Theorem \ref{thm:2cn}.
\begin{Corollary}
	Let $\Pc_1,\ldots,\Pc_m \subset \RR^N$ be lattice polytopes. Suppose that for each $i$, every edge of $\Pc_i$ has lattice length $\geq 2\dim(\Pc_i)(\dim(\Pc_i)+1)$.
	Then we obtain the following:
	\begin{enumerate}
		\item[{\rm (1)}]  $\Pc_1+\cdots+\Pc_m$ is IDP;
		\item[{\rm (2)}] 
		$\Pc_1*\cdots*\Pc_m$ is IDP if (and only if) $(\Pc_1,\ldots,\Pc_m)$ is IDP. 
	\end{enumerate}
\end{Corollary}
\begin{proof}
	This follows from Theorems \ref{thm:hh} and \ref{thm:2cn}.
\end{proof}

%\begin{Theorem}
%	Let $\Pc \subset \RR^N$ be a lattice polytope of degree $1$.
%	Then $\Pc$ possesses the integer decomposition property.
%\end{Theorem}

\section{Minkowski sums and Cayley sums of level polytopes}
In this section, we will give an answer to Question \ref{ques} (2).
First, we determine when the Minkowski sums and the Cayley sums of 2-convex-level lattice polytopes with interior lattice points are level.
\begin{Theorem}
	\label{thn:2cl}
	Let $\Pc_1,\ldots,\Pc_m \subset \RR^N$ be  2-convex-level lattice polytopes with  interior lattice points.
	Then we obtain the following:
	\begin{enumerate}
		\item[{\rm (1)}] The Minkowski sum $\Pc_1+\cdots+\Pc_m$ is level of index $1$;
		\item[{\rm (2)}] The Cayley sum $\Pc_1*\cdots*\Pc_m$ is level of index $m$.
	\end{enumerate}
\end{Theorem}
\begin{proof}
From Theorem \ref{thm;relation} (2), it is enough to show only the claim (2).
Since each $\Pc_i$ has interior lattice points, from Lemma \ref{decomp1}, $m(\Pc_1*\cdots*\Pc_m)$ has interior lattice points.
On the other hand, if for a positive integer $t$, $\xb$ is an interior lattice point in $t(\Pc_1*\cdots*\Pc_m)$, then $\xb$ belongs to $\{\sum_{i=1}^{m}a_i \eb_i\} \times \RR^{N}$ with positive integers $a_1,\ldots,a_m$.
Hence $t \geq m$.
This implies $\min \{t \in \ZZ_{\geq 1} : {\rm int}(t(\Pc_1*\cdots*\Pc_m)) \cap \ZZ^{m+N} \neq \emptyset \}=m$.
	Let $n \geq m+1$ and $\ab \in \text{int}(n(\Pc_1*\cdots* \Pc_m)) \cap \ZZ^{m+N}$.
	For $1 \leq i \leq m$, we let $\vb_{i1},\ldots,\vb_{ir_i}$ be the vertices of $\Pc_i$, where $r_i$ is the number of vertices of $\Pc_i$.
	Then since $\ab$ is an interior lattice point in $n(\Pc_1*\cdots*\Pc_m)$, we can write 
	$$\ab=\sum\limits_{i=1}^{m} \sum_{j=1}^{r_i}\lambda_{ij}(\eb_i,\vb_{ij}),$$
	where for any $1 \leq i \leq m$ and $1 \leq j \leq r_i$, $0<\lambda_{ij}$, and $\sum_{i=1}^{m} \sum_{j=1}^{r_i}\lambda_{ij}=n$.
	For each $1 \leq i \leq m$, set
	$t_i=\sum_{j=1}^{r_i}\lambda_{ij} \in \ZZ_{\geq 1}$ and $\ab'_i=\sum_{j=1}^{r_i}\lambda_{ij}(\eb_i,\vb_{ij}) \in \text{int}(t_i(\{\eb_i \} \times \Pc_i))$.
	Since each $\Pc_i$ is 2-convex-level, for each $i$, there exist $\bb_{i1} \in {\rm int}(\{\eb_i\} \times \Pc_i)$ and $\bb_2,\ldots,\bb_{t_i} \in (\{\eb_i \} \times \Pc_i) \cap \ZZ^{m+N}$ with $\ab'_i=\bb_{i1}+\bb_{i2}+\cdots+\bb_{it_i}$.
	Hence one has
	\[
	\ab=(\bb_{11}+\cdots+\bb_{m1})+\sum_{i=1}^m \sum_{j=2}^{t_i}\bb_{ij}.
	\]
 Then it follows that $\bb_{11}+\cdots+\bb_{m1} \in  {\rm int}(m (\Pc_1*\cdots*\cdots\Pc_m))$ and $\sum_{i=1}^m \sum_{j=2}^{t_i}\bb_{ij}\in (n-m)(\Pc_1*\cdots*\Pc_m) \cap \ZZ^{N+m}$.
 Since $\ab$ is a lattice point, $\bb_{11}+\cdots+\bb_{m1}$ must be a lattice point.
 Hence one has $\bb_{11}+\cdots+\bb_{m1}  \in {\rm int}(m (\Pc_1*\cdots*\Pc_m)) \cap \ZZ^{m+N}$.
	Therefore, $\Pc_1*\cdots*\Pc_m$ is level of index $m$, as desired.
\end{proof}

Next, we see when a dilated polytope is 2-convex-level.
\begin{Lemma}
	\label{lem:2cl}
	Let $\Pc \subset \RR^N$ be a lattice polytope.
	Then for any integer $n \geq \dim(\Pc)+1$,
$n\Pc$ has interior lattice points and is 2-convex-level.
\end{Lemma}

\begin{proof}
	Set $d=\dim(\Pc)$
	and
	$\ab \in \textnormal{int}(2n\Pc)$.
	Then there exist $d'+1$ affinely independent vertices $\vb_0,\ldots,\vb_{d'}$ of $\Pc$ such that $\ab=\sum_{i=0}^{d'}\lambda_i\vb_i$, where
	$d' \leq d$, for each $1 \leq i \leq d'$, $\lambda_i > 0$, and $\sum_{i=0}^{d'}\lambda_i=2n$.
	Since $2n \geq d+2 \geq d'+2$, there exists an index $j$ such that $\lambda_j > 1$.
	Hence $\ab=(\sum_{i=0}^{d}\lambda_i\vb_i-\vb_j)+\vb_j$ and $(\sum_{i=0}^{d}\lambda_i\vb_i-\vb_j) \in \textnormal{int}((2n-1)\Pc)$.
	From the fact that $n+1 \geq d+2$, $\ab$ can be written as
	$\ab=\ab'+\vb_{j_1}+\cdots+\vb_{j_n}$, where $\ab' \in \textnormal{int}(n\Pc)$.
	Hence $\ab \in \textnormal{int}(n\Pc)+(n\Pc \cap \ZZ^N)$, as desired.
\end{proof}
The integer $\dim(\Pc) +1$ in Lemma \ref{lem:2cl} is optimal.
\begin{Example}
	{\rm 
	Let $\Pc \subset \RR^2$ be the lattice triangle with vertices $(0,0),(1,0),(0,1)$. Then $2\Pc$ is not $2$-convex-level. Indeed, a point $\xb=(1,1) \in {\rm int} (4\Pc)$ cannot be written as $\xb=\ab+\bb$ with $\ab \in \Pc \cap \ZZ^2$ and $\bb \in {\rm int}(\Pc)$.
	Since $\dim \Pc=2$, $(\dim \Pc) \Pc$ is not $2$-convex-level. 
	}
\end{Example}

From Theorem \ref{thn:2cl} and Lemma \ref{lem:2cl} we obtain the following corollary.
\begin{Corollary}
	\label{cor:leveldi}
	Let $\Pc_1,\ldots,\Pc_m \subset \RR^N$ be lattice polytopes. For each $i$, let $n_i$ be a positive integer with $n_i \geq \dim (\Pc_i)+1$.
	Then we obtain the following:
	\begin{enumerate}
		\item[{\rm (1)}] {\rm (\cite{Hig})} $n_1\Pc_1+\cdots+n_m\Pc_m$ is level of index $1$;
		\item[{\rm (2)}] 
		$n_1\Pc_1*\cdots*n_m\Pc_m$ is level of index $m$. 
	\end{enumerate}
\end{Corollary}
Therefore, the answer to Question \ref{ques} (2) is positive.
Finally, we see that the bound $n_i \geq \dim (\Pc_i)+1$ in Corollary \ref{cor:leveldi} is optimal.
\begin{Example}
	{\rm 
	Let $\Pc_1,\Pc_2,\Pc_3 \subset \RR^3$ be the lattice polytopes with
	\begin{align*}
			\Pc_1&={\rm conv} \{ (0,0,0),(3,1,3)\}, 	\\
			\Pc_2&={\rm conv} \{ (0,0,0),(2,2,2)\}, 	\\
			\Pc_3&={\rm conv} \{ (0,0,0),(1,3,3)\}.
	\end{align*}
	Then each of $\Pc_1,\Pc_2,\Pc_3$ is of dimension $1$.
	Moreover, $\Pc_1+\Pc_2+\Pc_3$ is not Gorenstein with a unique interior lattice point, in particular, not level. Hence $\Pc_1 * \Pc_2*\Pc_3$ is not also level.
	}
\end{Example}

\end{document}